\documentclass[reqno]{amsart}%
\usepackage{amssymb,mathtools,enumerate,mathrsfs,lineno}

\usepackage{ifpdf}
\ifpdf
\usepackage[hyperindex]{hyperref}
\else
\expandafter\ifx\csname dvipdfm\endcsname\relax
\usepackage[hypertex,hyperindex]{hyperref}%
\else
\usepackage[dvipdfm,hyperindex]{hyperref}%
\fi
\fi

\allowdisplaybreaks[4]
\numberwithin{equation}{section}

\theoremstyle{plain}
\newtheorem{thm}{Theorem}[section]
\newtheorem{cor}{Corollary}[section]
\newtheorem{lem}{Lemma}[section]

\theoremstyle{remark}
\newtheorem{rem}{Remark}[section]

\theoremstyle{definition}
\newtheorem{dfn}{Definition}[section]

\DeclareMathOperator{\td}{d\!}
\DeclareMathOperator{\te}{e}

\begin{document}

\title[Bernstein and Stieltjes properties]
{Bernstein and Stieltjes properties of some functions involving logarithmic functions}

\author[F. Qi]{Feng Qi}

\address{School of Mathematics and Physics,
Hulunbuir University,
Hulunbuir 021008, Inner Mongolia,
China;
17709 Sabal Court, University Village,
Dallas, TX 75252-8024, USA}
\email{qifeng618@gmail.com}
\urladdr{https://orcid.org/0000-0001-6239-2968}



\begin{abstract}
We show that $\frac1{g(x)}$, with $g(x)=\frac{x-1}{x\ln x}$, is a complete Bernstein function and give its L\'evy measure explicitly. For $h(x)=\frac{(x^2-1)\ln(x+1)}{x^2\ln x}$, we prove that $h$ is not a Stieltjes function, so $\frac1h$ is not a complete Bernstein function, but $\frac1h$ is a Bernstein function. The latter rests on the positivity of a function built from the convolution of the exponential integrals $E_1$ and $E_2$. We also give a related completely monotonic function that is not a Stieltjes function.
\end{abstract}

\thanks{This paper was typeset using \AmS-\LaTeX}

\maketitle

\section{Preliminaries}
Let us review several foundational concepts to set the stage for what follows.

\begin{dfn}[{\cite[Chapter~XIII]{mpf-1993}, \cite{Schilling-Song-Vondracek-2nd}, and~\cite[Chapter~IV]{widder}}]
A function $f(x)$ is said to be completely monotonic on an interval $I\subseteq\mathbb{R}=(-\infty,\infty)$ if it is infinitely differentiable on $I$ and the inequalities
\begin{equation*}
(-1)^{n}f^{(n)}(x)\ge0, \quad n\in \mathbb{N}_0=\{0,1,2,\dotsc\}
\end{equation*}
hold for all $x\in I$.
\end{dfn}

\begin{dfn}[{\cite{absolute-mon-simp.tex, comgam1, e-gam-rat-comp-mon} and~\cite[Chapter~5]{Schilling-Song-Vondracek-2nd}}]
A positive function $f(x)$ is said to be logarithmically completely monotonic on an interval $I\subseteq\mathbb{R}$ if it is infinitely differentiable on $I$ and the inequalities
\begin{equation*}
(-1)^n[\ln f(x)]^{(n)}\ge 0, \quad n\in \mathbb{N}=\{1,2,\dotsc\}
\end{equation*}
hold for all $x\in I$.
\end{dfn}

\begin{dfn}[\cite{Schilling-Song-Vondracek-2nd, widder}]\label{Bernstein-dfn}
A function $f:I\subseteq\mathbb{R}\to\mathbb{R}_0=[0,\infty)$ is called a Bernstein function on the interval $I$ if $f(x)$ has derivatives of all orders and $f'(x)$ is completely monotonic on $I$.
\end{dfn}

\begin{dfn}[{\cite[p.~16, Definition~2.1]{Schilling-Song-Vondracek-2nd}}] \label{dfn-stieltjes}
A Stieltjes function is a function $f:\mathbb{R}_+=(0,\infty)\to\mathbb{R}_0$ which can be written in the form
\begin{equation*}
f(x)=\frac{a}x+b+\int_0^{\infty}\frac1{s+x}{\td\sigma(s)},
\end{equation*}
where $a,b\ge0$ are constants and $\sigma$ is a nonnegative measure on $\mathbb{R}_+$ such that
\begin{equation*}
\int_0^{\infty}\frac1{1+s}\td\sigma(s)<\infty.
\end{equation*}
\end{dfn}

\begin{dfn}[{\cite[Definition~6.1]{Schilling-Song-Vondracek-2nd}}]
If the L\'evy measure $\sigma$ from~\eqref{Levy-Khintchine-repreentation} has a completely monotonic density $m(t)$ with respect to the Lebesgue measure, that is, the integral representation
\begin{equation*}
f(x)=a+bx+\int_0^\infty\bigl(1-\te^{-xt}\bigr)m(t)\td t
\end{equation*}
is valid for $a,b\ge0$, where $m(t)$ is a completely monotonic function on $\mathbb{R}_+$ such that $\int_0^\infty\min\{1,t\}m(t)\td t<\infty$, then $f$ is said to be a complete Bernstein function on $\mathbb{R}_+$.
\end{dfn}

The above five kinds of functions have the following relations and properties.

\begin{lem}[{\cite{CBerg, absolute-mon-simp.tex, comgam1, JAAC384.tex, Schilling-Song-Vondracek-2nd} and~\cite[pp.~161\nobreakdash--162, Theorem~3]{Chen-Qi-Srivastava-09.tex}}] \label{relations-thm}
Every logarithmically completely monotonic function on an interval $I\subseteq\mathbb{R}$ is necessarily completely monotonic on $I$, but the converse does not hold.
\par
Every positive Stieltjes function is logarithmically completely monotonic on $\mathbb{R}_+$, but again the converse is not true.
\par
The reciprocal of any positive Bernstein function on an interval $I$ is logarithmically completely monotonic on $I$.
\end{lem}

\begin{lem}[{\cite[p.~457]{alzer-berg}, \cite[Remark~1.4]{CBerg}, \cite[Proposition~1.3]{Berg-Pedersen-JCAM-2001}, and~\cite[p.~507]{berg-pedersen-rocky}}] \label{Stieltjes-thm}
If $f(x)\ne0$ is a Stieltjes function, then so are the functions $\frac1{xf(x)}$ and $\frac{1}{f(\frac{1}{x})}$.
\end{lem}

\begin{lem}[{\cite[Theorem~3.2]{Schilling-Song-Vondracek-2nd}}]\label{Bernstein-thm3.2}
A function $f:\mathbb{R}_+\to\mathbb{R}_0$ is a Bernstein function if and only if it admits the representation
\begin{equation}\label{Levy-Khintchine-repreentation}
f(x)=a+bx+\int_0^\infty\bigl(1-\te^{-xt}\bigr)\td\sigma(t),
\end{equation}
where $a,b\ge0$ and $\sigma$ is a measure on $\mathbb{R}_+$ satisfying $\int_0^\infty\min\{1,t\}\td\sigma(t)<\infty$.
In particular, the triplet $(a,b,\sigma)$ determines $f$ uniquely and vice versa.
\end{lem}

\begin{lem}[{\cite[pp.~69--70, Theorem~6.2]{Schilling-Song-Vondracek-2nd}}]\label{Bern2nd-Thm6.2}
Suppose that $f$ is a non-negative function on $\mathbb{R}_+$. Then the following claims are equivalent.
\begin{enumerate}
\item
The function $f$ is a complete Bernstein function.
\item
The function $\frac{f(x)}{x}$ is a Stieltjes function.
\item
There exists a Bernstein function $g$ such that
\begin{equation*}
f(x)=x^2\mathscr{L}(g;x), \quad x>0,
\end{equation*}
where $\mathscr{L}$ stands for the Laplace transform.
\item
The function $f$ has an analytic continuation to $\mathbb{H}^\uparrow=\{z\in\mathbb{C}:\Im z>0\}$ such that $\Im f(z)\ge0$ for all $z\in\mathbb{H}^\uparrow$ and such that the limit $f(0+)=\lim_{\mathbb{R}_+\ni x\to0}f(x)$ exists and is real.
\item
The function $f$ has an analytic continuation to the cut complex plane $\mathbb{C}\setminus(-\infty,0]$ such that $(\Im z)[\Im f(z)]\ge0$ and the limit $f(0+)=\lim_{\mathbb{R}_+\ni x\to0}f(x)$ exists and is real.
\item
The function $f$ has an analytic continuation to $\mathbb{H}^\uparrow$ which is given by
\begin{equation*}
f(z)=a+bz+\int_{\mathbb{R}_+} \frac{z}{z+t}\td\sigma(t),
\end{equation*}
where $a,b\ge0$ are non-negative constants and $\sigma$ is a measure on $\mathbb{R}_+$ such that $\int_{\mathbb{R}_+} \frac{1}{1+t}\td\sigma(t)<\infty$.
\end{enumerate}
\end{lem}

\begin{lem}[{\cite[p.~93, Theorem~7.3]{Schilling-Song-Vondracek-2nd}}]\label{Bern2nd-Thm7.3}
A (non-trivial) function $f$ is a complete Bernstein function if and only if its reciprocal $\frac{1}{f}$ is a (non-trivial) Stieltjes function.
\end{lem}

\section{A new lemma}
This section establishes a positivity result for a function built from the convolution of the exponential integral functions $E_1$ and $E_2$; this result underlies the proof of Theorem~\ref{h(x)Bern-Thm}. Recall that the exponential integral $E_n(z)$ is defined in~\cite[p.~228, Entry~5.1.4]{abram} by
\begin{equation*}
E_n(z)=\int_1^{\infty} \frac{\te^{-sz}}{s^n}\, \td s,
\quad n\in\mathbb{N}_0,\quad \Re(z)>0.
\end{equation*}

\begin{lem}\label{G(t)-lem}
The function
\begin{equation*}
\mathfrak{G}(t)
=2\int_0^t \te^{\tau} E_2(\tau)E_1(t-\tau)\td\tau
-\bigl(\te^t-1\bigr)E_1(t)
\end{equation*}
is positive for $t\in\mathbb{R}_+$.
\end{lem}

\begin{proof}
It is clear that the exponential integral functions $E_1(x)$ and $E_2(x)$ are both positive and decreasing on $\mathbb{R}_+$. Hence,
$$
E_1(t-\tau)\ge E_1(t), \quad 0\le \tau<t.
$$
Accordingly,
\begin{multline*}
\int_0^t \te^{\tau} E_2(\tau)E_1(t-\tau)\td\tau
\ge E_1(t)\int_0^t \te^{\tau} E_2(\tau)\td\tau\\
\ge E_1(t)E_2(t)\int_0^t \te^{\tau}\td\tau
= E_1(t)E_2(t)\bigl(\te^t-1\bigr).
\end{multline*}
As a result,
$$
\mathfrak{G}(t)
=2\int_0^t \te^{\tau} E_2(\tau)E_1(t-\tau)\td\tau
-\bigl(\te^t-1\bigr)E_1(t)
\ge 2\bigl(\te^t-1\bigr)E_1(t)\biggl[E_2(t)-\frac{1}{2}\biggr].
$$
Since $E_2(t)$ is decreasing on $\mathbb{R}_+$, there exists a unique point $t_0\in\mathbb{R}_+$ such that
$E_2(t_0)=\frac{1}{2}$ and $E_2(t)\ge \frac{1}{2}$ for $t\in(0,t_0)$.
By numerical evaluations of the function $E_2(x)-x\ln x$ from $0.00$ to $0.50$ with step $0.01$ (see the first two columns of Table~5.4 in~\cite[p.~245]{abram}), we can therefore conclude that $E_2(0.26)=0.50744\dotsc$ and $E_2(0.27)=0.49744\dotsc$, which means that $t_0\in(0.26,\,0.27)$.
Consequently, the positivity $\mathfrak{G}(t)>0$ is valid for $t\in(0,0.26]$.
\par
Since $E_{n}'(x)=-E_{n-1}(x)$ for $n\in\mathbb{N}$ (see~\cite[p.~230, Entry~5.1.26]{abram}) and $E_2(0)=1$, integration by parts gives
\begin{align*}
&\quad\int_0^t \te^{\tau} E_2(\tau)E_1(t-\tau)\td\tau\\
&= \bigl[\te^{\tau} E_2(\tau)E_2(t-\tau)\bigr]\big|_{\tau=0}^{\tau=t}
+\int_0^t \te^{\tau} E_2(t-\tau)[E_1(\tau)-E_2(\tau)]\td\tau \\
&= \bigl(\te^t-1\bigr)E_2(t)
+\int_0^t \te^{\tau} E_2(t-\tau)[E_1(\tau)-E_2(\tau)]\td\tau.
\end{align*}
Thus,
$$
\mathfrak{G}(t)
= 2\bigl(\te^t-1\bigr)E_1(t)\biggl[\frac{E_2(t)}{E_1(t)}-\frac{1}{2}\biggr]
+2\int_0^t \te^{\tau} E_2(t-\tau)[E_1(\tau)-E_2(\tau)]\td\tau.
$$
Since $E_1(\tau)>E_2(\tau)$ (see~\cite[p.~229, Entry~5.1.17]{abram}) and $E_2(\tau)>0$ for $\tau\in\mathbb{R}_+$, we have
$$
2\int_0^t \te^{\tau} E_2(t-\tau)[E_1(\tau)-E_2(\tau)]\td\tau \ge 0,
\quad t\in\mathbb{R}_+.
$$
Since $\bigl[\frac{E_{n}(x)}{E_{n-1}(x)}\bigr]'>0$ for $x\in\mathbb{R}_+$ and $n\in\mathbb{N}$ (see~\cite[p.~229, Entry~5.1.21]{abram}), the ratio $\frac{E_2(x)}{E_1(x)}$ is increasing on $\mathbb{R}_+$. Therefore, to ensure $\mathfrak{G}(t)>0$ for all $t\in\mathbb{R}_+$, it suffices to verify that $\frac{E_2(t)}{E_1(t)} \ge \frac{1}{2}$ for $t\in[0.26,\infty)$.
Indeed, by numerical evaluations of the function $x^{-1}[E_1(x)+\ln x+\gamma]$ from $0.00$ to $0.50$ with step $0.01$ (see the first and fifth columns of Table~5.1 in~\cite[p.~238]{abram}), where $\gamma=0.57721566\dotsc$ stands for the Euler--Mascheroni constant, we can determine that $E_1(0.26)=1.0138\dotsc$, which leads to the conclusion, with the aid of the above evaluation $E_2(0.26)=0.50744\dotsc$, that
\begin{equation}\label{0.26ratio-half}
\frac{E_2(0.26)}{E_1(0.26)} = 0.5004\dotsm>\frac{1}{2}.
\end{equation}
Consequently, the positivity $\mathfrak{G}(t)>0$ is valid for all $t\in\mathbb{R}_+$. The proof of Lemma~\ref{G(t)-lem} is complete.
\end{proof}

\begin{rem}
The inequality~\eqref{0.26ratio-half} can be proved rigorously as follows. By~\cite[p.~229, Entries~5.1.11 and~5.1.14]{abram},
\begin{equation*}
E_1(x)=-\gamma-\ln x+\sum_{k=1}^{\infty}\frac{(-1)^{k+1}x^k}{k\cdot k!},
\qquad
E_2(x)=\te^{-x}-xE_1(x).
\end{equation*}
For $0<x\le1$ the terms $\frac{x^k}{k\cdot k!}$ decrease in $k$, so the alternating series is less than its partial sum ending with the positive term $\frac{x^3}{18}$. Using $\gamma>0.5772$ and $\ln0.26=\ln13+\ln2-2\ln10>-1.3471$, we obtain
\begin{equation*}
E_1(0.26)<0.7699+\biggl(0.26-\frac{0.26^2}{4}+\frac{0.26^3}{18}\biggr)<0.7699+0.24408<1.014.
\end{equation*}
Similarly, for $0<x\le1$ the terms $\frac{x^k}{k!}$ decrease in $k$, so $\te^{-x}>\sum_{k=0}^{5}\frac{(-x)^k}{k!}$, the partial sum ending with the negative term $-\frac{x^5}{120}$. For $x=0.26$ this equals $0.77105\dotsc$, hence $\te^{-0.26}>0.771$. Consequently
\begin{equation*}
E_2(0.26)=\te^{-0.26}-0.26\,E_1(0.26)>0.771-0.26\times1.014>0.507>\frac12
\end{equation*}
and
\begin{equation*}
\frac{E_2(0.26)}{E_1(0.26)}>\frac12 \iff E_1(0.26)<\frac{2\te^{-0.26}}{1+2\times0.26},
\end{equation*}
where the right-hand side exceeds 
$$
\frac{2\times0.771}{1.52}>1.0144>E_1(0.26).
$$
The inequality~\eqref{0.26ratio-half} thus holds. 
\end{rem}

\begin{rem}
The positivity of $\mathfrak{G}(t)$ on $\mathbb{R}_+$ was posed as a question at \url{https://math.stackexchange.com/q/5144702} (accessed on 23 July 2026). Two answers to this question were subsequently posted at \url{https://math.stackexchange.com/a/5144969} (accessed on 26 July 2026) and at \url{https://math.stackexchange.com/a/5144992} (accessed on 26 July 2026). The proof of Lemma~\ref{G(t)-lem} resolves this question.
\end{rem}

\section{Motivation and background}

In~\cite{22ICFIDCAA-Filomat.tex}, Qi and Li established integral representations of the functions 
$$
H(z),\quad \frac{1}{H(z)},\quad zH(z),\quad \frac{1}{zH(z)},\quad z^2H(z),\quad \frac{1}{z^2H(z)}
$$ 
for $z\in\mathbb{C}\setminus(-\infty,0]$, derived properties including (logarithmically) complete monotonicity, complete Bernstein, Stieltjes, and Thorin--Bernstein properties, and operator monotonicity, and computed the completely monotonic degrees, where
\begin{equation}\label{H-h-relation-eq}
H(z)=
\begin{dcases}
\frac{\ln z+\ln(1+z)-\ln(1+z^2)}{\ln(1+z^2)-\ln(1+z)},&z\ne1;\\
1,&z=1.
\end{dcases}
\end{equation}
\par
In~\cite[Remark~3]{Filomat1193.tex}, the function
\begin{equation}\label{h(x)}
h(x)=
\begin{dcases}
\frac{(x^2-1)\ln(x+1)}{x^2\ln x}, & x\in\mathbb{R}_+\setminus\{1\};\\
2\ln2, & x=1
\end{dcases}
\end{equation}
was conjectured to be logarithmically completely monotonic on $\mathbb{R}_+$.
In~\cite{MIA-9668.tex}, among other things, the above conjecture was confirmed as follows.

\begin{thm}[{\cite[Theorem~1]{MIA-9668.tex}}] \label{MIA-9668.thm}
For $x\in\mathbb{R}_+$, let
\begin{equation*}
g(x)=\begin{dcases}
\frac{x-1}{x\ln x}, & x\in\mathbb{R}_+\setminus\{1\};\\
1, & x=1.
\end{dcases}
\end{equation*}
Then both the functions $g(x)$ and $h(x)$ are logarithmically completely monotonic on $\mathbb{R}_+$.
Moreover, the function
\begin{equation*}
-\frac{g'(x)}{g(x)}=
\begin{dcases}
\frac{1}{x(x-1)}\biggl(\frac{x-1}{\ln x}-1\biggr), & x\in\mathbb{R}_+\setminus\{1\};\\
\frac{1}{2}, & x=1
\end{dcases}
\end{equation*}
has an integral representation
\begin{equation*}
-\frac{g'(x)}{g(x)}
=\frac{1}{x}\int_0^\infty\frac{1}{(\ln u)^2+\pi^2}\frac{1}{u+x}\td u, \quad x\in\mathbb{R}_+
\end{equation*}
and therefore the function $-\frac{g'(x)}{g(x)}$ is completely monotonic on $\mathbb{R}_+$ and the function $-x\frac{g'(x)}{g(x)}$ is a Stieltjes function.
\end{thm}

\begin{rem}
Combining Theorem~\ref{MIA-9668.thm} with Lemmas~\ref{Stieltjes-thm}, \ref{Bern2nd-Thm6.2}, and~\ref{Bern2nd-Thm7.3}, we conclude that 
\begin{enumerate}
\item
the functions
$$
\frac{1}{-x^2\frac{g'(x)}{g(x)}}
=
\begin{dcases}
\frac{(x-1) \ln x}{x (x-\ln x-1)}, & x\ne1;\\
2, & x=1
\end{dcases}
$$
and 
$$
\frac{1}{-\frac{1}{x}\frac{g'(1/x)}{g(1/x)}}
=
\begin{dcases}
\frac{(x-1) \ln x}{1-x+x \ln x}, & x\ne1;\\
2, & x=1
\end{dcases}
$$
are Stieltjes functions, and 
\item
the functions 
$$
-x^2\frac{g'(x)}{g(x)}=
\begin{dcases}
\frac{x (x-\ln x-1)}{(x-1) \ln x}, & x\ne1;\\
\frac{1}{2}, & x=1
\end{dcases}
$$ 
and 
$$
\frac{1}{-x\frac{g'(x)}{g(x)}}=
\begin{dcases}
\frac{(x-1) \ln x}{x-\ln x-1}, & x\ne1;\\
2, & x=1
\end{dcases}
$$
are complete Bernstein functions on $\mathbb{R}_+$.
\end{enumerate}
\end{rem}

\begin{rem}
In~\cite[Remark~1]{MIA-9668.tex}, the function $g(x)$ is stated to be a Stieltjes function, which follows from Definition~\ref{dfn-stieltjes} and the integral representation
\begin{equation*}
g(z)=\int_0^\infty \frac{t+1}{t[(\ln t)^2 +\pi^2]}\frac{\td t}{t+z},\quad z\in\mathbb{C}\setminus(-\infty,0]
\end{equation*}
established in~\cite[Lemma~2]{KMS-B14-0477.tex}. Accordingly, in view of Lemma~\ref{Stieltjes-thm}, we see that the function 
$$
\frac{1}{xg(x)}=\frac{1}{g(1/x)}=\frac{\ln x}{x-1}
$$ 
is also a Stieltjes function, which is stronger than the conclusion in~\cite[p.~19]{Schilling-Song-Vondracek-2nd}: the function $\frac{\ln(1+\lambda)}{\lambda}$ is a Stieltjes function. Furthermore, in light of Lemmas~\ref{Bern2nd-Thm6.2} and~\ref{Bern2nd-Thm7.3}, it follows that the functions
\begin{equation*}
xg(x)=g\biggl(\frac1x\biggr)=
\begin{dcases}
\frac{x-1}{\ln x}, & x\ne1;\\
1, & x=1
\end{dcases}
\end{equation*}
and
\begin{equation*}
\frac1{g(x)}=\begin{dcases}
\frac{x\ln x}{x-1}, & x\in\mathbb{R}_+\setminus\{1\};\\
1, & x=1
\end{dcases}
\end{equation*}
are complete Bernstein functions on $\mathbb{R}_+$.
\end{rem}

Based on Lemma~\ref{relations-thm} and Theorem~\ref{MIA-9668.thm}, it is natural to pose the following stronger questions:
\begin{enumerate}
\item
Is the function $h(x)$ a Stieltjes function?
\item
Is the reciprocal $\frac1{h(x)}$ a (complete) Bernstein function?
\end{enumerate}
\par
Addressing the two questions constitutes the main objective of this paper.

\section{Main results and their proofs}

We are now ready to present our main results and provide their proofs.
\par
We first provide a direct proof of the fact that the reciprocal $\frac{1}{g(x)}$ is a complete Bernstein function on $\mathbb{R}_+$.

\begin{thm}\label{g(x)Bern-Thm}
The reciprocal $\frac{1}{g(x)}$ is a complete Bernstein function on $\mathbb{R}_+$ and
\begin{equation*}
\frac{1}{g(x)}=\int_0^\infty\bigl(1-\te^{-xt}\bigr) \biggl[\frac{1}{t}\int_0^{1}\te^{(s-1) t/s}\td s\biggr]\td t, \quad x\in\mathbb{R}_+.
\end{equation*}
\end{thm}

\begin{proof}
By Definition~\ref{Bernstein-dfn}, we see that it is sufficient to prove the first derivative
\begin{equation*}
G(x)\triangleq\biggl[\frac{1}{g(x)}\biggr]'=
\begin{dcases}
\frac{x-1-\ln x}{(x-1)^2}, & x\ne1;\\
\frac{1}{2}, & x=1
\end{dcases}
\end{equation*}
is completely monotonic on $\mathbb{R}_+$.
\par
It is clear that
\begin{equation*}
G(x+1)=\begin{dcases}
\frac{x-\ln(1+x)}{x^2}, & x\ne0;\\
\frac{1}{2}, & x=0
\end{dcases}
\end{equation*}
for $x\in(-1,\infty)$ and it is easy to verify that
\begin{equation*}
G(x+1)=\int_0^{1}\frac{u}{1+xu}\td u, \quad x>-1.
\end{equation*}
As a result, the function $G(x+1)$ is completely monotonic on $(-1,\infty)$, that is, the function $G(x)$ is completely monotonic on $\mathbb{R}_+$. Consequently, the reciprocal $\frac{1}{g(x)}$ is a Bernstein function on $\mathbb{R}_+$.
\par
By the integral representation~\eqref{Levy-Khintchine-repreentation} in Lemma~\ref{Bernstein-thm3.2}, we see that
\begin{equation*}
a=\lim_{x\to0^+}\frac{1}{g(x)}=\lim_{x\to0^+}\frac{x\ln x}{x-1}=0
\end{equation*}
and
\begin{equation*}
b=\lim_{x\to\infty}\frac{1}{xg(x)}=\lim_{x\to\infty}\frac{\ln x}{x-1}=0.
\end{equation*}
\par
Further differentiating on both sides of~\eqref{Levy-Khintchine-repreentation} gives
\begin{equation*}
\biggl[\frac{1}{g(x)}\biggr]'=\int_0^\infty t\te^{-xt}\td\sigma(t).
\end{equation*}
This means that
\begin{multline}\label{G(x)Int}
\int_0^\infty t\te^{-xt}\td\sigma(t)=G(x)
=\int_0^{1}\frac{t}{1+(x-1)t}\td t
=\int_0^{1}\int_0^{\infty} \te^{(t-1) u/t} \te^{-u x}\td u\td t\\
=\int_0^{\infty} \biggl[\int_0^{1}\te^{(t-1) u/t}\td t\biggr] \te^{-u x}\td u
=\int_0^{\infty} \biggl[\int_0^{1}\te^{(s-1) t/s}\td s\biggr] \te^{-xt}\td t.
\end{multline}
By the uniqueness of Laplace transforms, we conclude that
\begin{equation*}
\td\sigma(t)=\int_0^{1}\te^{(s-1) t/s}\td s\frac{\td t}{t}.
\end{equation*}
Consequently, we arrive at
\begin{equation*}
\frac{1}{g(x)}=\frac{x\ln x}{x-1}=\int_0^\infty\bigl(1-\te^{-xt}\bigr) \biggl[\frac{1}{t}\int_0^{1}\te^{(s-1) t/s}\td s\biggr]\td t, \quad x\in\mathbb{R}_+.
\end{equation*}
Since both $\frac{1}{t}$ and $\te^{(s-1)t/s}$ for $s\in(0,1)$ are completely monotonic on $t>0$, the function $\frac{1}{t}\int_0^{1}\te^{(s-1)t/s}\td s$ is also completely monotonic on $t>0$. Consequently, $\frac{1}{g(x)}$ is a complete Bernstein function on $\mathbb{R}_+$.
The proof of Theorem~\ref{g(x)Bern-Thm} is thus complete.
\end{proof}

\begin{rem}
The function $G(x+1)$ can be reformulated as the normalized remainder
\begin{equation*}
T_1[\ln(x+1)]=2G(x+1).
\end{equation*}
For detailed information about the normalized remainders, please refer to the review and survey article~\cite{GM-VG-QiF-Ch.tex}.
\par
In general, since
\begin{equation*}
\ln(1+x)=\sum_{k=1}^{\infty}(-1)^{k+1}\frac{x^k}{k}, \quad |x|<1,
\end{equation*}
the $n$th normalized remainder for $n\in\mathbb{N}$ is
\begin{equation*}
T_n[\ln(1+x)]=
\begin{dcases}
(-1)^{n}\frac{n+1}{x^{n+1}} \Biggl[\ln(1+x)-\sum_{k=1}^{n}(-1)^{k+1}\frac{x^k}{k}\Biggr], & x\ne0;\\
1, & x=0.
\end{dcases}
\end{equation*}
Is the normalized remainder $T_n[\ln(1+x)]$ for all $n\in\mathbb{N}$ a completely monotonic function on the interval $(-1,\infty)$? The answer is yes, and we will supply a streamlined proof in a series of subsequent papers.
\end{rem}

\begin{thm}\label{h-not-Stiel-thm}
The function $h(x)$ defined by~\eqref{h(x)} is not a Stieltjes function.
\end{thm}

\begin{proof}
If $h(x)$ were a Stieltjes function, then, in view of Lemma~\ref{Stieltjes-thm}, the function
$$
\frac{1}{h\bigl(\frac{1}{x}\bigr)}=\frac{\ln x}{(x^2-1) [\ln(1+x)-\ln x]}
$$
should also be a Stieltjes function. In particular, it should be completely monotonic on $\mathbb{R}_+$. But its first derivative
$$
F(x)\triangleq\biggl[\frac{1}{h\bigl(\frac{1}{x}\bigr)}\biggr]'=\frac{\bigl(x^2-2 x^2 \ln x-1\bigr) \ln\bigl(\frac{1}{x}+1\bigr)+(x-1) \ln x}{x (x^2-1)^2 \ln^2\bigl(\frac{1}{x}+1\bigr)}
$$
is not increasing on $\mathbb{R}_+$, due to
\begin{multline*}
F\biggl(\frac{1}{20}\biggr)=\frac{8000[2(190+\ln21) \ln20-399 \ln21]}{159201 (\ln21)^2}
=-0.31522\dotsc\\
>F\biggl(\frac{1}{10}\biggr)
=\frac{1000[2 (45+\ln11)\ln10-99 \ln11]}{9801 (\ln11)^2}=-0.33921\dotsc.
\end{multline*}
Consequently, the function $h(x)$ is not a Stieltjes function. The proof of Theorem~\ref{h-not-Stiel-thm} is thus complete.
\end{proof}

Since $h(x)$ is not a Stieltjes function, by Lemma~\ref{Bern2nd-Thm7.3}, we deduce that $\frac{1}{h(x)}$ is surely not a complete Bernstein function on $\mathbb{R}_+$. But the reciprocal $\frac{1}{h(x)}$ is a Bernstein function on $\mathbb{R}_+$, a stronger result than the logarithmically complete monotonicity of $h(x)$ in Theorem~\ref{MIA-9668.thm}.

\begin{thm}\label{h(x)Bern-Thm}
The reciprocal $\frac{1}{h(x)}$ is a Bernstein function on $\mathbb{R}_+$.
\end{thm}

\begin{proof}
By Definition~\ref{Bernstein-dfn}, we see that it is sufficient to prove the first derivative
\begin{align*}
\biggl[\frac{1}{h(x)}\biggr]'
&=\frac{x\bigl[\bigl(x^2-2 \ln x-1\bigr) \ln(x+1)-(x-1) x \ln x\bigr]}{(x^2-1)^2 \ln^2(x+1)}\\
&=\biggl[\frac{x}{(x+1)\ln(x+1)}\biggr]^2 \biggl[\frac{x^2-2 \ln x-1}{(x-1)^2} \frac{\ln(x+1)}{x}-\frac{\ln x}{x-1}\biggr]
\end{align*}
is completely monotonic on $\mathbb{R}_+$.
\par
In view of Theorem~\ref{MIA-9668.thm}, the function $$\frac{x}{(x+1)\ln(x+1)}=g(x+1)$$ is logarithmically completely monotonic on $(-1,\infty)$. Since the product of finitely many logarithmically completely monotonic functions remains logarithmically completely monotonic on the intersection of their logarithmically completely monotonic intervals (which can be verified straightforwardly), it follows that $$\biggl[\frac{x}{(x+1)\ln(x+1)}\biggr]^2 =[g(x+1)]^2$$ is also logarithmically completely monotonic on $(-1,\infty)$. Hence, by virtue of Lemma~\ref{relations-thm}, this function is completely monotonic on $(-1,\infty)$. In the foregoing arguments, the point $x=0$ is a removable singularity.
\par
It remains to prove the complete monotonicity of the function
\begin{align*}
\frac{x^2-2 \ln x-1}{(x-1)^2} \frac{\ln(x+1)}{x}-\frac{\ln x}{x-1}
&=\biggl[1+\frac{2(x-1-\ln x)}{(x-1)^2}\biggr] \frac{\ln(x+1)}{x}-\frac{\ln x}{x-1}\\
&=[1+2G(x)] \frac{\ln(x+1)}{x}-\frac{\ln x}{x-1}\\
&\triangleq\mathcal{G}(x)
\end{align*}
on $\mathbb{R}_+$, where $x=1$ is a removable singularity.
\par
Using the Frullani integral
\begin{equation*}
\ln\frac{b}{a}=\int_0^{\infty}\frac{\te^{-au}-\te^{-bu}}{u}\td u, \quad a,b>0,
\end{equation*}
in~\cite[p.~230, Entry~5.1.32]{abram}, we derive that
\begin{multline*}
\frac{\ln(1+x)}{x}=\frac{1}{x} \int_0^\infty \frac{\te^{-s}-\te^{-(x+1)s}}{s} \td s \\
=\int_0^\infty \frac{\te^{-s}}{s} \biggl(\int_0^s \te^{-xt} \td t \biggr) \td s
=\int_0^\infty \biggl(\int_t^\infty \frac{\te^{-s}}{s} \td s \biggr) \te^{-xt} \td t.
\end{multline*}
By virtue of this integral representation and the integral representation~\eqref{G(x)Int}, we deduce that
\begin{equation*}
1+2G(x)=1+2\int_0^{\infty} \biggl[\int_0^{1}\te^{(s-1) t/s}\td s\biggr] \te^{-xt}\td t.
\end{equation*}
Accordingly, we obtain
\begin{align*}
\mathcal{G}(x)&=\biggl(1+2\int_0^{\infty} \biggl[\int_0^{1}\te^{(s-1) t/s}\td s\biggr] \te^{-xt}\td t\biggr)\int_0^\infty \biggl(\int_t^\infty \frac{\te^{-s}}{s} \td s \biggr) \te^{-xt} \td t \\
&\quad-\int_0^\infty \biggl(\int_t^\infty \frac{\te^{-s}}{s} \td s \biggr) \te^{-(x-1)t} \td t\\
&=\int_0^\infty \biggl(\int_t^\infty \frac{\te^{-s}}{s} \td s \biggr) \te^{-xt} \td t-\int_0^\infty \biggl(\int_t^\infty \frac{\te^{-s}}{s} \td s \biggr) \te^{-(x-1)t} \td t\\
&\quad+2\int_0^{\infty} \biggl[\int_0^{1}\te^{(s-1) t/s}\td s\biggr] \te^{-xt}\td t\int_0^\infty \biggl(\int_t^\infty \frac{\te^{-s}}{s} \td s \biggr) \te^{-xt} \td t,
\end{align*}
where, by the Laplace convolution theorem~\cite[p.~1020, Entry~29.2.8]{abram}, 
\begin{multline*}
\int_0^{\infty} \biggl[\int_0^{1}\te^{(s-1) t/s}\td s\biggr] \te^{-xt}\td t\int_0^\infty \biggl(\int_t^\infty \frac{\te^{-s}}{s} \td s \biggr) \te^{-xt} \td t\\
=\int_0^{\infty} \biggl[\int_0^t \biggl(\int_0^{1} \te^{(s-1)\tau/s} \td s \biggr)
\biggl(\int_{t-\tau}^{\infty} \frac{\te^{-u}}{u} \td u \biggr) \td\tau\biggr]\te^{-xt}\td t.
\end{multline*}
As a result, we arrive at
\begin{multline*}
\mathcal{G}(x)=
\int_0^{\infty}\biggl[\bigl(1-\te^t\bigr)\int_t^\infty \frac{\te^{-s}}{s} \td s\\
+2\int_0^t \biggl(\int_0^{1} \te^{(s-1)\tau/s} \td s \biggr)
\biggl(\int_{t-\tau}^{\infty} \frac{\te^{-u}}{u} \td u \biggr) \td\tau\biggr]\te^{-xt}\td t,
\end{multline*}
where
\begin{align*}
\int_0^{1} \te^{(s-1)\tau/s} \td s
=\int_0^{1} \te^{(1-1/s)\tau} \td s
&=\te^\tau\int_0^{1} \te^{-\tau/s} \td s\\
&=\te^\tau\tau\int_{\tau}^{\infty} \frac{\te^{-s}}{s^2} \td s
=\te^\tau\int_1^{\infty} \frac{\te^{-s\tau}}{s^2} \td s.
\end{align*}
Therefore, we acquire
\begin{equation}
\begin{aligned}\label{G(x)INT}
\mathcal{G}(x)
&=\int_0^{\infty}\biggl[2\int_0^t\te^\tau E_2(\tau)E_1(t-\tau)\td\tau-\bigl(\te^t-1\bigr)E_1(t)\biggr]\te^{-xt}\td t\\
&\triangleq\int_0^{\infty}\mathfrak{G}(t)\te^{-xt}\td t.
\end{aligned}
\end{equation}
From Lemma~\ref{G(t)-lem}, it follows immediately that the function $\mathcal{G}(x)$ is completely monotonic on $\mathbb{R}_+$.
Since the function $[g(x+1)]^2$ is completely monotonic on $(-1,\infty)\supset\mathbb{R}_+$ and the product of two completely monotonic functions is completely monotonic, the derivative
$$
\biggl[\frac{1}{h(x)}\biggr]'=[g(x+1)]^2\,\mathcal{G}(x)
$$
is completely monotonic on $\mathbb{R}_+$.
Since $\frac{1}{h(x)}$ is also nonnegative on $\mathbb{R}_+$, the function $\frac{1}{h(x)}$ is a Bernstein function.
The proof of Theorem~\ref{h(x)Bern-Thm} is thus complete.
\end{proof}

\begin{cor}
The function
\begin{equation*}
\mathcal{G}(x)=
\begin{dcases}
\biggl[1+\frac{2(x-1-\ln x)}{(x-1)^2}\biggr] \frac{\ln(x+1)}{x}-\frac{\ln x}{x-1}, & x\ne1;\\
2\ln2-1, & x=1
\end{dcases}
\end{equation*}
is completely monotonic on $\mathbb{R}_+$ and has the Laplace integral representation~\eqref{G(x)INT}; but it is not a Stieltjes function.
\end{cor}

\begin{proof}
The complete monotonicity follows from the proof of Theorem~\ref{h(x)Bern-Thm}.
\par
If $\mathcal{G}(x)$ were a Stieltjes function, then, by Lemma~\ref{Stieltjes-thm}, the function $\frac{1}{x\mathcal{G}(x)}$ should also be a Stieltjes function, and then $\frac{1}{x\mathcal{G}(x)}$ should be a completely monotonic function on $\mathbb{R}_+$, but, since $\frac{1}{x\mathcal{G}(x)}=2.587\dotsc$ at $x=2$ and $\frac{1}{x\mathcal{G}(x)}=3.170\dotsc$ at $x=5$, it is not decreasing on the whole positive axis $\mathbb{R}_+$. This contradiction reveals that the function $\mathcal{G}(x)$ is not a Stieltjes function.
\end{proof}

\begin{rem}
We conjecture that the function $\mathcal{G}(x)$ is logarithmically completely monotonic on $\mathbb{R}_+$.
\end{rem}

\begin{rem}
By the integral representation~\eqref{Levy-Khintchine-repreentation} in Lemma~\ref{Bernstein-thm3.2}, we see that
\begin{equation*}
a=\lim_{x\to0^+}\frac{1}{h(x)}=\lim_{x\to0^+}\frac{x^2\ln x}{(x^2-1)\ln(x+1)}=0
\end{equation*}
and
\begin{equation*}
b=\lim_{x\to\infty}\frac{1}{xh(x)}=\lim_{x\to\infty}\frac{x\ln x}{(x^2-1)\ln(x+1)}=0.
\end{equation*}
Further differentiating on both sides of~\eqref{Levy-Khintchine-repreentation} gives
\begin{equation*}
\biggl[\frac{x}{(x+1)\ln(x+1)}\biggr]^2 \mathcal{G}(x)
=\int_0^\infty t\te^{-xt}\td\sigma(t).
\end{equation*}
Can one obtain an explicit expression of the L\'evy measure $\sigma(t)$ in~\eqref{Levy-Khintchine-repreentation} for the Bernstein function $f(x)=\frac{1}{h(x)}$?
\end{rem}

\begin{rem}
Since the reciprocal $\frac{1}{h(x)}$ is not a complete Bernstein function on $\mathbb{R}_+$, then, by Lemmas~\ref{Bern2nd-Thm6.2} and~\ref{Bern2nd-Thm7.3}, the functions
\begin{equation}\label{Conj1}
\frac{1}{xh(x)}=
\begin{dcases}
\frac{x\ln x}{(x^2-1)\ln(x+1)}, & x\ne1;\\
\frac{1}{2\ln2}, & x=1
\end{dcases}
\end{equation}
and $h(x)$ are not Stieltjes functions. 
\par
The fact that the function $h(x)$ is not a Stieltjes function has been confirmed by Theorem~\ref{h-not-Stiel-thm}.
\par
If $\frac{1}{xh(x)}$ were a Stieltjes function, then, by Lemma~\ref{Stieltjes-thm}, the functions $h(x)$ and
\begin{equation}\label{Conj2}
\frac{1}{x}h\biggl(\frac{1}{x}\biggr)=
\begin{dcases}
\frac{(x^2-1)[\ln(x+1)-\ln x]}{x \ln x}, & x\ne1;\\
2\ln2, & x=1
\end{dcases}
\end{equation}
should also be Stieltjes functions. This contradicts Theorem~\ref{h-not-Stiel-thm}. Consequently, the function $\frac{1}{xh(x)}$ is not a Stieltjes function.
\par
We conjecture that the functions in~\eqref{Conj1} and~\eqref{Conj2} are completely monotonic, even logarithmically completely monotonic, on $\mathbb{R}_+$.
\end{rem}

\section{Declarations}

\paragraph{\bf Author Contributions}
The author was solely responsible for the conception, analysis, and preparation of the manuscript.

\paragraph{\bf Funding}
The author was partially supported by the Natural Science Foundation of Inner Mongolia Autonomous Region (Grant No.~2025QN01041) and by the Youth Project of Hulunbuir City for Basic Research and Applied Basic Research (Grant No.~GH2024020).

\paragraph{\bf Institutional Review Board Statement}
Not applicable.

\paragraph{\bf Informed Consent Statement}
Not applicable.

\paragraph{\bf Ethical Approval}
The conducted research is not related to either human or animal use.

\paragraph{\bf Availability of Data and Materials}
Data sharing is not applicable to this article as no new data were created or analyzed in this study.

\paragraph{\bf Acknowledgements}
Not applicable.

\paragraph{\bf Competing Interests}
The author declares no competing interests.

\paragraph{\bf Use of AI tools declaration}
The author declares that no AI tools were used in the creation of this article.


\begin{thebibliography}{99}

\bibitem{abram}
M. Abramowitz and I. A. Stegun (Eds.), \textit{Handbook of Mathematical Functions with Formulas, Graphs, and Mathematical Tables}, National Bureau of Standards, Applied Mathematics Series \textbf{55}, Reprint of the 1972 edition, Dover Publications, Inc., New York, 1992.

\bibitem{alzer-berg}
H. Alzer and C. Berg, \textit{Some classes of completely monotonic functions}, Ann. Acad. Sci. Fenn. Math. \textbf{27} (2002), no.~2, 445\nobreakdash--460.

\bibitem{CBerg}
C. Berg, \textit{Integral representation of some functions related to the gamma function}, Mediterr. J. Math. \textbf{1} (2004), no.~4, 433\nobreakdash--439. DOI: \url{https://doi.org/10.1007/s00009-004-0022-6}.

\bibitem{Berg-Pedersen-JCAM-2001}
C. Berg and H. L. Pedersen, \textit{A completely monotone function related to the gamma function}, J. Comput. Appl. Math. \textbf{133} (2001), 219\nobreakdash--230. DOI: \url{https://doi.org/10.1016/S0377-0427(00)00644-0}.

\bibitem{berg-pedersen-rocky}
C. Berg and H. L. Pedersen, \textit{Pick functions related to the gamma function}, Rocky Mountain J. Math. \textbf{32} (2002), 507\nobreakdash--525. DOI: \url{https://doi.org/10.1216/rmjm/1030539684}.

\bibitem{Chen-Qi-Srivastava-09.tex}
C.-P. Chen, F. Qi, and H. M. Srivastava, \textit{Some properties of functions related to the gamma and psi functions}, Integral Transforms Spec. Funct. \textbf{21} (2010), no.~2, 153\nobreakdash--164. DOI: \url{https://doi.org/10.1080/10652460903064216}.

\bibitem{absolute-mon-simp.tex}
B.-N. Guo and F. Qi, \textit{A property of logarithmically absolutely monotonic functions and the logarithmically complete monotonicity of a power-exponential function}, Politehn. Univ. Bucharest Sci. Bull. Ser. A Appl. Math. Phys. \textbf{72} (2010), no.~2, 21\nobreakdash--30.

\bibitem{Filomat1193.tex}
X.-L. Liu, T. Xu, C.-Y. He, and F. Qi, \textit{Sharp bounds for even-indexed Bernoulli numbers and their ratios}, Filomat \textbf{40} (2026), accepted on 14 June 2026. URL: \url{https://www.researchgate.net/publication/407043457}.

\bibitem{mpf-1993}
D. S. Mitrinovi\'c, J. E. Pe\v{c}ari\'c, and A. M. Fink, \textit{Classical and New Inequalities in Analysis}, Kluwer Academic Publishers, Dordrecht-Boston-London, 1993. DOI: \url{https://doi.org/10.1007/978-94-017-1043-5}.

\bibitem{GM-VG-QiF-Ch.tex}
F. Qi, \textit{Theory of normalized remainders in Taylor series expansions}, arXiv:2512.15797v3. DOI: \url{https://doi.org/10.48550/arXiv.2512.15797}. Chapter~12 in Contemporary Approaches to Approximation of Operators---Approximation Theory and Mathematical Methods, edited by Gradimir V. Milovanovi\'c and Vijay Gupta, published in Birkh\"auser's \href{https://link.springer.com/series/4961}{Trends in Mathematics} series, pp.~349\nobreakdash--400, 2026. DOI: \url{https://doi.org/10.1007/978-3-032-38767-7_12}.

\bibitem{comgam1}
F. Qi and C.-P. Chen, \textit{A complete monotonicity property of the gamma function}, J. Math. Anal. Appl. \textbf{296} (2004), 603\nobreakdash--607. DOI: \url{https://doi.org/10.1016/j.jmaa.2004.04.026}.

\bibitem{e-gam-rat-comp-mon}
F. Qi, B.-N. Guo, and C.-P. Chen, \textit{Some completely monotonic functions involving the gamma and polygamma functions}, J. Aust. Math. Soc. \textbf{80} (2006), 81\nobreakdash--88. DOI: \url{https://doi.org/10.1017/S1446788700011393}.

\bibitem{JAAC384.tex}
F. Qi and W.-H. Li, \textit{A logarithmically completely monotonic function involving the ratio of gamma functions}, J. Appl. Anal. Comput. \textbf{5} (2015), no.~4, 626\nobreakdash--634. DOI: \url{https://doi.org/10.11948/2015049}.

\bibitem{22ICFIDCAA-Filomat.tex}
F. Qi and W.-H. Li, \emph{Integral representations and properties of some functions involving the logarithmic function}, Filomat \textbf{30} (2016), no.~7, 1659\nobreakdash--1674. DOI: \url{https://doi.org/10.2298/FIL1607659Q}.

\bibitem{KMS-B14-0477.tex}
F. Qi and X.-J. Zhang, \emph{An integral representation, some inequalities, and complete monotonicity of the Bernoulli numbers of the second kind}, Bull. Korean Math. Soc. \textbf{52} (2015), no.~3, 987\nobreakdash--998. DOI: \url{https://doi.org/10.4134/BKMS.2015.52.3.987}.

\bibitem{Schilling-Song-Vondracek-2nd}
R. L. Schilling, R. Song, and Z. Vondra\v{c}ek, \textit{Bernstein Functions---Theory and Applications}, 2nd ed., de Gruyter Studies in Mathematics \textbf{37}, Walter de Gruyter, Berlin, Germany, 2012. DOI: \url{https://doi.org/10.1515/9783110269338}.

\bibitem{widder}
D. V. Widder, \textit{The Laplace Transform}, Princeton University Press, Princeton, 1946.

\bibitem{MIA-9668.tex}
T.-Y. Zhang, A.-P. Ji, and F. Qi, \textit{Two logarithmically completely monotonic functions involving logarithmic function}, Math. Inequal. Appl. \textbf{29} (2026), no.~3, 313\nobreakdash--318. DOI: \url{https://doi.org/10.7153/mia-2026-29-18}.

\end{thebibliography}
\end{document}